\documentclass[11pt]{article}
\usepackage[T1]{fontenc}
\usepackage{lmodern}
\usepackage[margin=1.15in]{geometry}
\usepackage{amsmath,amssymb,amsthm}
\usepackage[hidelinks]{hyperref}
\hypersetup{
  pdftitle={A counterexample to the claw-free Schur-positivity conjecture},
  pdfauthor={Jitendra Prajapati}
}
\newtheorem{theorem}{Theorem}
\newtheorem{lemma}[theorem]{Lemma}
\newtheorem{proposition}[theorem]{Proposition}

\theoremstyle{remark}
\newtheorem{remark}[theorem]{Remark}
\newcommand{\csf}{X}

\title{A counterexample to the claw-free Schur-positivity conjecture}
\author{Jitendra Prajapati\thanks{Independent Researcher. Email:
\href{mailto:jitendraprajapati2603@gmail.com}{\nolinkurl{jitendraprajapati2603@gmail.com}}. Verification code:
\url{https://github.com/infinityscroll/claw-free-schur-counterexample}.}}
\date{July 2026}

\begin{document}
\maketitle

\begin{abstract}
The claw-free Schur-positivity conjecture, recorded by Stanley (1998) and
credited there to Gasharov, asserts that the chromatic symmetric function of
every claw-free graph is Schur-positive. We give a counterexample
on $12$ vertices: the line graph $G$ of the graph obtained from a
$4$-cycle by attaching triangles at two opposite vertices and pendant edges
at the other two satisfies $[s_{(3,3,3,3)}]\,\csf_G = -64$. The coefficient
follows from a short computation by hand and is also reproduced by three exact
implementations. An exhaustive computation over all
$216{,}777$ connected claw-free graphs on at most $11$ vertices shows that
every one is Schur-positive, so $12$ vertices is the minimum order of any
counterexample. A complete census of the $1{,}728{,}404$ connected claw-free
graphs on $12$ vertices finds exactly two non-Schur-positive isomorphism
classes; the other has graph6 code
\texttt{K?\textasciigrave CR@\textasciigrave bAbRB} and coefficient
$[s_{(3,3,3,3)}]=-40$.
\end{abstract}

\section{Introduction}

For a finite simple graph $G = (V, E)$, the chromatic symmetric function of
Stanley \cite{Stanley95} is
\[
\csf_G \;=\; \sum_{\kappa} \prod_{v \in V} x_{\kappa(v)},
\]
the sum over all proper colourings $\kappa \colon V \to \{1, 2, \dots\}$.
Gasharov \cite{Gasharov96} proved that the incomparability graph of any
$(3+1)$-free poset is Schur-positive. Stanley
\cite[Conjecture~1.4]{Stanley98}, crediting Gasharov, conjectured that Schur
positivity extends to all claw-free graphs, i.e.\ graphs with no induced
$K_{1,3}$. Subsequent work established further special cases
\cite{Gasharov99}, including, recently, generalized nets \cite{SW25}.

We show the conjecture is false.

\emph{Independent work and chronology.} The example in
Theorem~\ref{thm:main} was
\href{https://github.com/infinityscroll/claw-free-schur-counterexample/commit/7747970833d8d4ce2dadba8f89a184cb52d814a9}
{released in the public verification repository} on July 22, 2026.
Matherne and Morales submitted an independently obtained
preprint on July 23, 2026, giving two counterexamples \cite{MM26}. Their two
graphs are isomorphic to the two graphs classified here: each pair contains
two distinct order-$12$ counterexamples, and Proposition~\ref{prop:minimal}
shows that exactly two such isomorphism classes exist. The independence was
publicly confirmed by the authors; the example was also independently
verified by D.~Grinberg and others \cite{MO}.

\begin{theorem}\label{thm:main}
Let $H$ be the graph obtained from a $4$-cycle $abcda$ by attaching a
triangle at $a$, a triangle at $c$, and pendant edges at $b$ and $d$, and
let $G = L(H)$ be its line graph ($12$ vertices, $22$ edges, graph6 code
\texttt{K?\textasciigrave CRAWWUXIM}). Then $G$ is connected and claw-free,
and
\[
[s_{(3,3,3,3)}]\,\csf_G \;=\; -64 .
\]
\end{theorem}

Since $G$ is a line graph, the theorem also disproves the corresponding
restriction of the conjecture to line graphs, even when the root graph has
maximum degree four.

\section{Proof of Theorem~\ref{thm:main}}

Write the edges of $H$ as
$ab, bc, cd, da$ (the $4$-cycle), $au, av, uv$ (the triangle at $a$),
$cx, cy, xy$ (the triangle at $c$), $b\ell$ and $dm$ (the pendant edges).
Line graphs are claw-free: of any three edges of $H$ meeting a fixed edge
$e$, two share an endpoint of $e$. Connectivity of $G$ follows from
connectivity of $H$.

\subsection{Stable partitions}

Stable (independent) sets of $G = L(H)$ are exactly matchings of $H$, so
the coefficient $[m_\lambda] \csf_G$ counts ordered partitions of $E(H)$
into matchings with sizes $\lambda_1, \lambda_2, \dots$

\begin{lemma}
The matching number of $H$ is $4$.
\end{lemma}

\begin{proof}
$\{uv,\, xy,\, b\ell,\, dm\}$ is a matching of size $4$. Since $H$ has ten
vertices, a matching of size $5$ would be perfect. It would have to contain
the two leaf edges $b\ell$ and $dm$. Removing their endpoints leaves the two
disjoint triangles on $\{a,u,v\}$ and $\{c,x,y\}$, neither of which has a
perfect matching. Thus no matching of size $5$ exists.
\end{proof}

Consequently every stable partition of $V(G)$ has parts of size at most
$4$, hence at least $\lceil 12/4 \rceil = 3$ parts, and every monomial
coefficient $[m_\lambda]\csf_G$ with $\lambda_1 \ge 5$ vanishes.

\subsection{Counting the four-part stable partitions}

Since the vertex $a$ has degree $4$ in $H$, in any partition of $E(H)$ into
$4$ matchings the four edges $ab, ad, au, av$ lie in four distinct parts;
normalizing their labels to $1, 2, 3, 4$ picks exactly one labelled
representative from each unordered partition. The four edges at the other
degree-$4$ vertex $c$ receive a permutation $p$ of $\{1,2,3,4\}$ on
$(bc, cd, cx, cy)$ subject to $p_1 \ne 1$ (as $bc$ meets $ab$ at $b$) and
$p_2 \ne 2$ (as $cd$ meets $ad$ at $d$); there are
$24 - 6 - 6 + 2 = 14$ such $p$. Given $p$, each of the four remaining
edges $b\ell, dm, uv, xy$ has exactly two admissible labels, independently.
Sorting the class sizes of the resulting $14 \cdot 16 = 224$ labelled
partitions gives the following distribution.
\begin{center}
\begin{tabular}{c|rrr}
$p$ & $(4,4,2,2)$ & $(4,3,3,2)$ & $(3,3,3,3)$ \\ \hline
$2134$ or $2143$ (each) & 4 & 8 & 4 \\
each of the other $12$ & 2 & 12 & 2 \\ \hline
total & 32 & 160 & 32
\end{tabular}
\end{center}
No partition of type $(4,4,4)$ or $(4,4,3,1)$ occurs. With
$[m_\lambda]\csf_G = N_\lambda \prod_j m_j(\lambda)!$, where $N_\lambda$ is
the number of unordered stable partitions of type $\lambda$ and
$m_j(\lambda)$ the multiplicity of $j$ in $\lambda$:
\[
[m_{4422}] = 32 \cdot 2! \cdot 2! = 128, \quad
[m_{4332}] = 160 \cdot 2! = 320, \quad
[m_{3333}] = 32 \cdot 4! = 768 .
\]

\subsection{Triangular inversion}

Since $s_\nu = \sum_\lambda K_{\nu\lambda} m_\lambda$ with the Kostka matrix
unitriangular with respect to dominance, and since all monomial
coefficients with $\lambda_1 \ge 5$ vanish, the only Schur coefficients that
can be nonzero are indexed by partitions with parts at most $4$; and the
coefficients $[s_\nu]\csf_G$ for
$\nu \in \{(4,4,4), (4,4,3,1)\}$ vanish because the corresponding monomial
coefficients (and all higher ones) do. Using
$K_{(4,4,2,2),(4,3,3,2)} = 1$, $K_{(4,4,2,2),(3,3,3,3)} = 2$,
$K_{(4,3,3,2),(3,3,3,3)} = 3$:
\[
[s_{4422}] = 128, \qquad
[s_{4332}] = 320 - 1 \cdot 128 = 192,
\]
\[
[s_{3333}] = 768 - 2 \cdot 128 - 3 \cdot 192 = -64 . \qquad \qed
\]

\begin{remark}
The full Schur expansion of $\csf_G$ (all $77$ coefficients, computed
exactly by three separately implemented routes) has $(3,3,3,3)$ as its only
negative coefficient.
\end{remark}

\section{Verification and minimality}

The coefficient $[s_{(3,3,3,3)}]\csf_G = -64$, and the fact that it is the
only negative Schur coefficient, have been reproduced by three separately
implemented exact routes: normalized edge colourings and semistandard
tableau enumeration; power-sum inclusion--exclusion; and a stable-partition
dynamic program followed by exact Kostka inversion. The implementations
include self-tests on complete graphs, the claw, and brute-force colouring
counts.

\begin{proposition}[computational]\label{prop:minimal}
Every connected claw-free graph on at most $11$ vertices has Schur-positive
chromatic symmetric function. Among the connected claw-free graphs on $12$
vertices, exactly two isomorphism classes are not Schur-positive: the graph
in Theorem~\ref{thm:main} and the graph $G'$ with graph6 code
\texttt{K?\textasciigrave CR@\textasciigrave bAbRB}. Their negative
coefficients are respectively
$[s_{(3,3,3,3)}]\csf_G=-64$ and
$[s_{(3,3,3,3)}]\csf_{G'}=-40$, and neither graph has another negative
Schur coefficient. Consequently these are exactly the two counterexamples
of minimum order.
\end{proposition}

The computation enumerated all connected graphs on $1 \le n \le 12$
vertices with \textsf{nauty}'s \texttt{geng}. The connected-graph counts
match \href{https://oeis.org/A001349}{OEIS A001349} exactly.
\begin{center}
\small
\begin{tabular}{c|r|r|r}
order & connected & claw-free & non-Schur-positive \\ \hline
$\leq 8$ & $12{,}113$ & $1{,}145$ & $0$ \\
$9$ & $261{,}080$ & $4{,}494$ & $0$ \\
$10$ & $11{,}716{,}571$ & $26{,}389$ & $0$ \\
$11$ & $1{,}006{,}700{,}565$ & $184{,}749$ & $0$ \\
$12$ & $164{,}059{,}830{,}476$ & $1{,}728{,}404$ & $2$ \\ \hline
total & $165{,}078{,}520{,}805$ & $1{,}945{,}181$ & $2$
\end{tabular}
\end{center}
Every Schur coefficient was computed exactly in integer arithmetic.
Disconnected graphs of order at most $12$ reduce to smaller connected
components, since chromatic symmetric functions multiply under disjoint
union and Schur positivity is closed under products.

\section*{Computational provenance}
The counterexample was found by an exhaustive computer search. The graph,
short proof, exact verification programs, and census logs are available in
the
\href{https://github.com/infinityscroll/claw-free-schur-counterexample}
{public verification repository}.

\end{document}